\documentclass[12pt]{amsart}

\usepackage[colorlinks = true,
            linkcolor = red,
            urlcolor  = blue,
            citecolor = blue,
            anchorcolor = blue]{hyperref}

\numberwithin{equation}{section}

\newcommand{\ben}{\begin{enumerate}}
\newcommand{\een}{\end{enumerate}}

\newcommand{\bea}{\begin{eqnarray}}
\newcommand{\ba}{\begin{array}}
\newcommand{\bean}{\begin{eqnarray*}}
\newcommand{\ea}{\end{array}}
\newcommand{\eea}{\end{eqnarray}}
\newcommand{\eean}{\end{eqnarray*}}
\newcommand{\beq}{\begin{equation}}
\newcommand{\eeq}{\end{equation}}
\newcommand{\bthm}{\begin{thm}}
\newcommand{\ethm}{\end{thm}}
\newcommand{\blem}{\begin{lem}}
\newcommand{\elem}{\end{lem}}
\newcommand{\bprop}{\begin{prop}}
\newcommand{\eprop}{\end{prop}}
\newcommand{\bcor}{\begin{cor}}
\newcommand{\ecor}{\end{cor}}
\newcommand{\bdfn}{\begin{dfn}}
\newcommand{\edfn}{\end{dfn}}
\newcommand{\brem}{\begin{rem}}
\newcommand{\erem}{\end{rem}}
\newcommand{\bpf}{\begin{proof}}
\newcommand{\epf}{\end{proof}}
\newcommand{\bfact}{\begin{fact}}
\newcommand{\efact}{\end{fact}}
\newcommand{\bobs}{\begin{obs}}
\newcommand{\eobs}{\end{obs}}

\alph{enumii} \roman{enumiii}

\newtheorem{thm}{Theorem}[section]
\newtheorem{prop}[thm]{Proposition}
\newtheorem{lem}[thm]{Lemma}

\newtheorem{cor}[thm]{Corollary}
\newtheorem{dfn}[thm]{Definition}
\newtheorem{rem}[thm]{Remark}
\newtheorem{fact}[thm]{Fact}
\newtheorem{obs}[thm]{Observation}

\def\N{{\mathbb N}}                \def\Z{{\mathbb Z}}      \def\R{{\mathbb R}}
\def\C{{\mathbb C}}                  
\def\Q{{\mathbb Q}}
\def\1{1\!\!1}
\def\and{\text{ and }}

        \def\diam{\text{\rm {diam}}}

      \def\lra{\longrightarrow}

\def\h{\rm{h}}
\def\hmu{\h_\mu}

\def\Int{\text{{\rm Int}}}
         \def\P{\text{{\rm P}}}

                   \def\Pa{{\mathcal P}}

\def\a{\alpha}                \def\b{\beta}             \def\d{\delta}
               \def\vp{\varepsilon}          
                           \def\l{\lambda}
              \def\om{\omega}           \def\Om{\Omega}
               \def\sg{\sigma}

\def\bi{\bigcap}              
\def\({\bigl(}                \def\){\bigr)}
\def\lt{\left}                \def\rt{\right}

\def\ld{\ldots}                        \def\^{\tilde}

\def\es{\emptyset}            \def\sms{\setminus}

      \def\imp{\Rightarrow}

\def\ov{\overline}

\def\om{\omega}

\newcommand{\Ups}{\Upsilon}

\def\Q{\mathbb Q}

\begin{document}

\title[]
{ \bf\large {Invariant measures in  non-conformal fibered systems with singularities}}
\date{\today}

\author[\sc Eugen MIHAILESCU]{\sc Eugen Mihailescu}

\address{Eugen Mihailescu,
Institute of Mathematics of the Romanian Academy,
P.O Box 1-764,
RO 014700, Bucharest, Romania}
\email{Eugen.Mihailescu@imar.ro    \ \hspace*{0.5cm}
Web: www.imar.ro/$\sim$mihailes}

\date{}
\thanks{Research  supported in part by grant PN-III-P4-ID-PCE-2020-2693 from UEFISCDI}
\maketitle
\begin{abstract}
 
We study invariant measures and thermodynamic formalism for a class of endomorphisms $F_{T}$ which are only piecewise differentiable on countably many pieces and non-conformal. 
The endomorphism  $F_{T}$ has parametrized countably generated limit sets $J_{T, \om}$ in stable fibers. We prove  a Global Volume Lemma for $F_T$ implying that the projections of equilibrium measures  are exact dimensional on a non-compact global basic set $J_T $. A dimension formula for these global measures is obtained by using the  Lyapunov exponents and marginal entropies.  Then, we study the equilibrium measures of geometric potentials $\psi_{T, s}$, and we prove that the dimensions of the associated measures $\nu^\om_s$ in fibers are independent of $\om$ and they depend real-analytically on the parameter $s$ from an interval $\mathcal F(T)$. Moreover, we establish a Variational Principle for dimension in fibers. 
\end{abstract}

\

\textbf{MSC 2010:} 28D05, 37D35, 37A35, 37C45, 37A44, 46G10, 60A10.

\textbf{Keywords:} Pressure functional;  equilibrium measures; entropy; pointwise dimension;  Lyapunov exponents of measures; marginal entropies of measures; Variational Principles. 

\section{Introduction and Outline.}

In this paper we  explore a large class of invariant measures for a class $\mathcal E$ of endomorphisms with singularities  $F_T: (1, \infty)^2 \times Y\to (1, \infty)^2 \times Y$ which are \textbf{piecewise differentiable} on countably many pieces, non-uniformly hyperbolic and non-conformal, where $Y$ is the closure of a bounded open set in $\mathbb C$, and $T$ is a type of skew-product endomorphism.  The fact that we have countably many pieces of differentiability for $F_T$ makes this case very different from the case of only finitely many such pieces. Moreover the non-conformality of $F_T$ and the presence of parabolic points present additional difficulties.  The thermodynamic formalism of the transformations  $F_T$ presents interesting features and has connections to ergodic number theory; we introduce several  new ideas and methods for their study.

The first coordinate map of $F_T$ takes values in $ (1, \infty) \times (1, \infty)$, it is non-conformal and only piecewise differentiable on countably many open sets, and has parabolic points.   The second coordinate map of $F_T$ consists of  parametrized conformal contractions, and we obtain a family of countably generated \textbf{fiber  limit sets} $J_{T, \omega}$ in the stable fibers, where the parameter $\omega$ belongs to  a 1-sided shift  space $\Sigma_I^+$ with countable alphabet $I$. 
 We prove first that the projections of the conditional measures of equilibrium states from the 2-sided shift space $\Sigma_I$, are exact dimensional on the fibers $J_{T, \omega}$ and we find their dimensions.

Then, we study the subtle metric properties for the global invariant measures $\nu$ on the  \textbf{non-compact global basic sets} $J_T\subset (1, \infty)^2 \times Y \subset \mathbb C^2$, formed by taking the  union of  fiber limit sets of type $\{z\}\times J_{T, \om}$, for $z = \tilde \pi(\om)$, and $\tilde \pi: \Sigma_I^+ \to (1, \infty) \times (1, \infty)$ is a coding map and $I = \N^* \times \N^*$. We show that the measures $\nu$ are \textbf{exact dimensional} on these non-compact global basic sets $J_T \subset (1, \infty)^2 \times Y$, thus proving a type of  Eckmann-Ruelle Conjecture for them (see \cite{ER}). 
Moreover, we find the general formula for the Hausdorff dimension of $\nu$ (and thus for its pointwise and box dimensions), using the  Lyapunov exponents and marginal entropies.

In particular, we investigate the case of equilibrium measures for geometric potentials $\psi_{T, s}$. We show that the dimensions of their associated fiber measures depends real-analytically on the parameter $s$ from an interval $\mathcal F(T)$, and we establish a Variational Principle for Hausdorff dimension on the fractal set $J_{T, \om}$.

For some examples of $T$ our method generates representations as certain multi-dimensional continued fractions for the points in the limit sets. 

\

Exact dimensionality is an important property for a measure, as shown for eg in \cite{BPS}, \cite{LY}, \cite{Pe}, \cite{Y}.  However, our case here is different from the  papers above, due to the lack of differentiability everywhere, and to the fact that we work on a non-compact manifold with measures that are not compactly supported.

In general, the ergodic theory of  endomorphisms and dimensions of invariant sets and measures were studied in many cases, for eg in \cite{ABN},  \cite{BPS}, \cite{FM},  \cite{LY}, \cite{Ma1}, \cite{MM},  \cite{M-MZ}, \cite{M-stable}, \cite{MS},  \cite{Pa}, \cite{Pe}, \cite{Ro}, \cite{Ru82}, \cite{Ru}, 
  \cite{Y}, to mention several contributions in this direction.   
 In \cite{Pa} Parry studied endomorphisms from an ergodic points of view, including  Jacobians for invariant measures.
In \cite{Ma1} Manning showed that for an Axiom A diffeomorphism of a surface preserving an ergodic measure $\mu$, the entropy $h(\mu)$ is equal to the product of the positive Lyapunov exponent of $\mu$ and the dimension of the set of generic points in an unstable manifold.
In \cite{Ru82} Ruelle expressed the Hausdorff dimension of an invariant repeller as the zero of a pressure function. In \cite{Y} Young found a formula  for the pointwise dimension of a hyperbolic measure $\mu$ (i.e $\mu$ has only non-zero Lyapunov
exponents) invariant to a smooth diffeomorphism of a surface; this proves also the exact dimensionality of  $\mu$. 
In \cite{LY} Ledrappier and Young proved a formula for the entropy of an invariant measure $\mu$ for a diffeomorphism of a compact Riemannian manifold, using the Lyapunov exponents and the dimensions of $\mu$ in the directions of the hyperbolic subspaces.  Then in \cite{Pe} Pesin studied the Carath\'eodory-Pesin structures  and their applications in dimension theory. 
  For smooth diffeomorphisms on compact manifolds without boundary, Barreira, Pesin and Schmeling solved in \cite{BPS} the Eckmann-Ruelle Conjecture (\cite{ER}); they proved a type of local product structure for invariant hyperbolic measures $\mu$ and computed the dimension of $\mu$ as the sum of the stable and the unstable pointwise
dimensions.  
Dimensions for finite non-conformal function systems were studied for eg by Falconer in \cite{Fa1} and, in the case of hyperbolic measures on compact manifolds by Ledrappier and Young \cite{LY}.  In \cite{FH} Feng and Hu proved the exact dimensionality of self-conformal measures for finite iterated function systems with  overlaps, and Feng extended  to finite self-affine systems in \cite{F}. In \cite{A} Allaart studied self-affine functions and the multifractal
formalism for self-similar measures for function systems satisfying the open set condition.  
And in \cite{M-stable} were studied dimensions for conditional measures on stable manifolds for a class of hyperbolic endomorphisms. 

The dimension theory for countable iterated function systems and their invariant measures presents significant differences from the case of finite systems; it had contributions  by Mauldin and Urba\'nski in \cite{MaU-ifs}, \cite{gdms}, by Mihailescu and Urba\'nski in \cite{MU-Adv}, \cite{MU-ETDS}, and by other authors. 
Also, various types of multi-dimensional continued fractions and relations with ergodic theory were investigated for example in \cite{Sch} and the references therein. 

\

In the current paper, we study a new class $\mathcal E$ of endomorphisms $F_T$, which are only piecewise non-uniformly hyperbolic  on countably many pieces in non-compact manifolds and non-conformal, and which have a different kind of dynamics. For instance our measures $\nu$ on the basic set $J_T$ of $F_T$ are not compactly supported. Our setting is thus  different from the case of hyperbolic diffeomorphisms on compact manifolds studied in \cite{BPS}, \cite{LY}, \cite{Y}. 

\

We define the \textbf{endomorphisms} $F_T$  in the following way:

\

 Consider first  a conformal Smale skew-product endomorphism $T: \Sigma_I^+ \times Y \to \Sigma_I^+ \times Y, T(\om, y) = (\sigma(\om), T_\om(y))$, as in \cite{MU-ETDS} (see Definition \ref{confSmale}) with alphabet $E = I = \N^*\times \N^*$, where $Y$ is the closure of a bounded open set in $\C$. In fibers we have conformal injective contractions $T_\om: Y \to Y, \    \om \in \Sigma_I^+.$ Some examples of such endomorphisms $T$ are given in the next Section.
 
Next, consider  the representation in continued fractions for irrational coordinates of points in $(1, \infty) \times (1, \infty)$,  $\tilde \pi: \Sigma_I^+ \to (1, \infty) \times (1, \infty)$,
\begin{equation}\label{pitilde}
 \tilde \pi(\omega) = \tilde \pi\big(((m_0, n_0), (m_1, n_1), \ldots)\big) := \pi_1(\omega) + i \pi_2(\omega), \ \omega \in \Sigma_I^+,
 \end{equation}
 where $$\pi_1(\omega) := m_0 + \frac{1}{m_1 + \frac{1}{m_2+\ldots}}, \  \text{and} \ \  \pi_2(\omega) := n_0+ \frac{1}{n_1+\frac{1}{n_2+\ldots}}.$$
 
Then, given the conformal Smale endomorphism $T$ and the representation $\tilde \pi$ as above, the new endomorphism $F_T$ is defined on the open set $(1, \infty) \times (1, \infty) \times Y \ \subset \C^2$, by $$F_T: (1, \infty) \times (1, \infty) \times Y \lra (1, \infty) \times (1, \infty) \times Y,$$ 
\begin{equation}\label{generalF}
F_T(z, w) = \begin{cases} \Big(\frac{1}{\{Re(z)\}} +  \frac{i}{\{Im(z)\}}, \ T_\om(w) \Big),  &\text{if}  \ Re(z) \notin \Q,   Im (z) \notin \Q,  \text{and} \ z = \tilde \pi(\om),\\
\\
(2, 2, \frac 12) \hspace{1.3in}, &\text{if} \ Re (z) \in \mathbb Q \ \text{or} \ Im (z) \in \mathbb Q,
\end{cases} 
\end{equation}
where the fractional part of a positive number $x$ is $\{x\} := x - [x]$.
Thus $F_T$ has singularities at all points $(z, w) \in \C^2$ with $Re(z) \in \Q$ or $Im(z) \in \Q$.

Then, $\mathcal E$ denotes the \textbf{collection} of endomorphisms $F_T$ for all conformal Smale skew-products $T$ (namely Definition \ref{E}).

\

The thermodynamic formalism of the maps $F_T$ is very different from the one for Smale endomorphisms $T$. This is due to the countably many domains of differentiability and the non-conformality and  non-uniform hyperbolicity of $F_T$. We introduce in the sequel some new methods and ideas for their study, which are different from those in \cite{MU-ETDS}.

\

\ Our \textbf{main results} in the sequel are: 

\

In \textbf{Section 2}, we recall some notions and results about Smale skew-products $T$  and introduce the  class $\mathcal E$ consisting of countably piecewise hyperbolic endomorphisms $F_T$. 
In Theorem \ref{exactT} we prove the exact dimensionality of the projections of conditional measures on the fibers $J_{T, \om}$ for general endomorphisms $F_T$.

\

In \textbf{Sections 3 and 4} we give the \textbf{main new results} and proofs of the paper. 
\newline
In Theorem \ref{modexT} we study the \textbf{general case} of a skew-product endomorphism $F_T$  associated to a conformal Smale skew-product $T$.
If the complex coordinates of the map $F_T$ are written as $F_T = (F_{T, 1}, F_{T, 2})$, then $F_{T, 1}$ is a piecewise differentiable and piecewise hyperbolic map  on countably many pieces and non-conformal on these pieces.
We prove a Global Volume Lemma for the projection measures $\nu$.  This is  used then to show that the global invariant  measures $\nu$ are \textbf{exact dimensional} on  the non-compact \textbf{global basic set} $J_T \subset \C^2$.  Moreover, in Theorem \ref{modexT} we find the \textbf{general formula} for the Hausdorff dimension of the global measures $\nu$, which involves the Lyapunov exponents and marginal entropies. 

The proof of  Theorem \ref{modexT}  introduces several new ideas and methods.  We  deal with the fact that the first complex coordinate map $F_{T, 1}$ is only piecewise differentiable on countably many pieces and it is not conformal on these pieces. Moreover, the measure $\nu$ can have non-equal Lyapunov exponents along the  axes  in the $z$-plane, and is not compactly supported.  The second coordinate map $F_{T, 2}$ is conformal and the iterates coming from countably many $n$-preimages generate parametrized limit sets $J_{T, \om}$ in the $w$-fibers. All these facts require a proof with several steps of the global Volume Lemma. 

\

Then in \textbf{Section 4}, we give a class of examples by using the equilibrium measures of \textbf{geometric potentials} $\psi_{T, s}$. In Theorem \ref{geoms} we prove that the dimensions $\delta_{T, s}$ of the associated projection  measures in fibers, depend \textbf{real-analytically} on the parameter $s$ in an interval which depends on $T$. 

In Theorem \ref{indepcx} we prove that the dimension of the set $J_{T, \om}$ is independent of $\om \in \Sigma_I^+$. And then in Theorem \ref{varprin} we establish a \textbf{Variational Principle for dimension} on $J_{T, \om}$. 

\

\section{Background notions and results.}

Let us now recall several notions that will be used in the sequel. 
Firstly, the notion of \textit{pointwise dimension} for a measure, and that of \textit{exact dimensional measure} (for eg \cite{Pe}). 
For exact dimensional measures,  the pointwise, Hausdorff, and box dimensions (see for eg \cite{Fa} for definitions) will all coincide (\cite{Y}). 

\bdfn\label{EDim}
Let a probability measure $\mu$ on a metric space $X$, and for a point $x \in X$ define the upper pointwise dimension, respectively the lower pointwise dimension of $\mu$ at $x$ by  $$\overline\delta(\mu)(x) = \mathop{\limsup}\limits_{r\to 0} \frac{\log \mu(B(x, r))}{\log r},\ \text{and} \ \underline\delta(\mu)(x) = \mathop{\liminf}\limits_{r\to 0} \frac{\log \mu(B(x, r))}{\log r}.$$
If $\overline\delta(\mu)(x) = \underline\delta(\mu)(x)$, then we say that the pointwise dimension of $\mu$ exists at $x$, and we denote it by $\delta(\mu)(x)$. 
A measure $\mu$ is called exact dimensional on  $X$ if the pointwise dimension of $\mu$ exists for $\mu$-a.e $x\in X$ and $\delta(\mu)(\cdot)$ is constant $\mu$-a.e., denoted by $\delta(\mu)$.
\edfn

\

Another notion is that of Smale skew-product endomorphism. Below are briefly recalled some notions/results from \cite{MU-ETDS}. Let us mention that in this case there exist several significant differences from the case of shifts over finite alphabets. For example, the topological entropy of the shift over a countable infinite alphabet is infinite. 

We recall now some results from thermodynamic formalism of 2-sided shifts $(\Sigma_E, \sigma)$ on countable alphabets $E$.
For $\b>0$ the metric $d_\b$
on $E^\Z$ is
$$
d_\b\((\om_n)_{-\infty}^{+\infty},(\tau_n)_{-\infty}^{+\infty}\)
=\exp\(-\b\max\{n\ge 0:\forall_{k\in\Z}|k|\le n\, \imp\, \om_k=\tau_k\}\)
$$
with  $e^{-\infty}=0$. All
metrics $d_\b$, $\b>0$, on $E^\Z$ induce the product topology on $E^\Z$. We set
$$
\Sigma_E = \big\{(\om_n)_{-\infty}^{+\infty}, \om_n \in E, \forall n\in\Z \}
$$
H\"older continuity is defined similarly as before, for potentials $\psi: \Sigma_E \to \R$.
For any $\om\in \Sigma_E$ and $-\infty\le m\le n\le+\infty$, define 
the truncation between the $m$ and $n$ positions as $
\om|_m^n =\om_m \om_{m+1}\ld\om_n.
$
If
$\tau\in \Sigma_E$, let the cylinder from $m$ to $n$ positions,
$
[\tau]_m^n=\{\om\in \Sigma_E, \om|_m^n=\tau|_m^n\}
$.
The family of cylinders from $m$ to $n$ is
denoted by $\mathcal C_m^n$. 

Let $\psi: \Sigma_E \to \R$ be a continuous function. Then the \textit{topological pressure}
$\P(\psi)$ is,
\beq\label{1_2015_11_04}
\P(\psi)
:=\lim_{n\to\infty}\frac1n\log\sum_{[\om]\in \mathcal C_0^{n-1}} \exp\(\sup\(S_n\psi|_{[\om]}\)\),
\eeq
where the limit above exists by subadditivity. 
 A shift-invariant
Borel probability $\mu$ on $\Sigma_E$ is called a \textit{Gibbs measure}
of $\psi$ if there are constants $C\ge 1$, $P\in\R$
such that
\beq\label{1082305}
C^{-1}
\le {\mu([\om|_0^{n-1}]) \over \exp(S_n\psi(\om)-Pn)}
\le C
\eeq
for all $n\ge 1, \om\in \Sigma_E$.  From (\ref{1082305}), it follows that if $\psi$ has a Gibbs
state, then automatically $P=\P(\psi)$.  As before, a function $\psi:\Sigma_E \to\R$ is called \textit{summable} if:
$$
\sum_{e\in E}\exp\(\sup\(\psi|_{[e]}\)\)<\infty.
$$

\bthm\label{eqmesi}
A H\"older continuous $\psi$ is summable if and only if
$P(\psi) < \infty$. For every H\"older continuous summable potential $\psi:
\Sigma_E \to\R$ 
there exists a unique Gibbs state $\mu_\psi$ on $\Sigma_E$, and 
$\mu_\psi$ is ergodic. 

Also the following Variational Principle for pressure holds,
$$
P(\psi) = \sup\lt\{\hmu(\sg)+\int_{\Sigma_E} \psi d\mu, \text{for}  \ \mu \ \sigma-\text{invariant probability on} \ \Sigma_E, \int_{\Sigma_E} \psi d\mu>-\infty\rt\},
$$
and $\mu_\psi$ is the only measure at which this supremum is attained.
\ethm


Consider now the partition of the space  $\Sigma_E$ with the infinite cylinders determined by the non-negative indices,
$$
\Pa_-=\{[\eta|_0^{\infty}]:\eta \in \Sigma_E\}
     =\{[\om]:\om\in \Sigma_E^+\}.
$$
 $\Pa_-$ is a measurable partition of $\Sigma_E$. If $\mu$
is a Borel probability on $\Sigma_E$, let the Rokhlin canonical system of conditional measures associated to  $\Pa_-$ (see \cite{Pa}, \cite{Ro}), denoted by
$
\{\ov\mu^\tau:\tau\in \Sigma_E\}
$.
Then $\ov\mu^\tau$ is a probability
measure on the cylinder $[\tau|_0^{\infty}]$ and we denote by $\ov\mu^\om$, $\om\in \Sigma_E^+$,  the 
conditional measure on $[\om]$. The truncation to non-negative indices is:
$$
\pi_0: \Sigma_E \to \Sigma_E^+, \  \pi_0(\tau) = \tau|_0^\infty, \tau \in \Sigma_E,
$$
The system $\{\ov\mu^\om, \om\in \Sigma_E^+\}$ of conditional measures is uniquely determined up to measure zero by the property (see \cite{Ro}) that, for all $g\in L^1(\mu)$, 
$$
\int_{\Sigma_E}g\ d\mu=\int_{\Sigma_E^+}\int_{[\om]}g\ d\ov\mu^{\om}
  \ d(\mu\circ \pi_0^{-1})(\om).
$$
 
 \
 
Define  now in our case the \textit{Smale conformal skew-product endomorphisms}.

\begin{dfn}\label{confSmale} (\cite{MU-ETDS})
Let $Y$ a closed bounded set in $ \C$, $E$ a countable alphabet, and 
assume:
\begin{itemize}

\item[(a)] The interior of $Y$ is nonempty, and let for each $\om \in \Sigma_E^+$, $T_\om: Y \to Y$ be a continuous injective map. Denote the map  $T:\Sigma_I^+ \times Y \to \Sigma_I^+ \times Y, T(\om, y) = (\sigma(\om), T_\om(y))$. 

\item[(b)] Each map $T_\om:Y \to Y$ extends to a $C^1$ conformal embedding from $Y^*$ to $Y^*$, where $Y^*$ is a bounded connected open subset of $\C$ containing $Y$. Then $T_\om$ denotes also this extension and assume that the maps
$
T_\om:Y^*\to Y^*$ satisfy:
\item[(c)] There exists $\l>1$ such 
that  for all $\om\in \Sigma_E^+$ and all $y_1,y_2\in Y^*$, 
\beq\label{1111205}
d(T_\om(y_2),T_\om(y_1))\le \l^{-1}d(y_2,y_1)
\eeq
\item[(d)] (Bounded Distortion Property 1) There are constants $\a > 0, H>0$ s.t $\forall y, z\in Y^*$, \ 
$$
\big|\log|T_\om'(y)|-\log|T_\om'(z)|\big|\le H||y-z||^\a.
$$
\item[(e)] The function 
$
\Sigma_E\ni\tau\longmapsto\log|T_\tau'(\hat\pi_2(\tau))|\in\R
$
is H\"older continuous.
\item[(f)] (Open Set Condition) For every $\om\in \Sigma_E^+$ and for all $a, b\in E$ with  $a\ne b$, we have 
$
T_{a\om}(\Int(Y))\cap T_{b\om}(\Int(Y))=\es.
$
\item[(g)] (Strong Open Set Condition) There exists a measurable function $\d:\Sigma_E^+\to(0,\infty)$, so that for every $\om \in \Sigma_E^+$, \ 
$
J_{T, \om}\cap\(Y\sms\ov B(Y^c, \d(\om)\)\ne\es.
$
\end{itemize}
Then $T$ is called a Smale conformal skew-product endomorphism. 
 
 \end{dfn}
 
 \
 
 Below are some \textbf{Examples} of Smale conformal skew-product endomorphisms:
 
 \
 
\  \textbf{1)} Let $Y= \overline{B(1/2, 1/2)} \subset \C$ and for every $\om \in \Sigma_I^+$ define $$T_\om(z) = \frac{1}{\bar z+\tilde \pi(\om)}, z \in Y.$$ If $z \in Y$, then its complex conjugate $\bar z$ is also in $Y$, and we use the beautiful fact that the image of the set $\{z \in \C, Re(z) \ge 1\}$ through the map $z \to \frac 1z$ is equal to $Y$. Then $T_\om$ is a well-defined injective conformal map. From the fact that $\tilde\pi(\om) = m_0+\frac{1}{m_1+\frac{1}{\ldots}} + i(n_0 + \frac{1}{n_1+\frac{1}{\ldots}}), \ \om = ((m_0, n_0), (m_1, n_1), \ldots) \in \Sigma_I^+$,  it follows that if  $a, b \in I = \N^*\times \N^*, a \ne b$ then $|\bar z - \bar z'| < |\tilde\pi(a\om) - \tilde\pi(b\om)|$ for any $z, z' \in B(1/2, 1/2)$, hence $$T_{a\om}(B(1/2, 1/2)) \cap T_{b\om}(B(1/2, 1/2)) = \emptyset.$$ 
 
 Due to the expression of the representation $\tilde\pi:\Sigma_I^+ \to (1, \infty)^2$ and since for any $z\in Y$, $$1/4 \le |T_\om'(z)| \le 4, \forall \om \in \Sigma_I^+,$$ we see that the Bounded Distortion Property above is satisfied.  Also the maps $T_\om$ are uniformly contracting on $Y$ for $\om \in \Sigma_I^+$.  Moreover, the map $\tau \to \hat\pi_2(\tau) $ is H\"older continuous on $\Sigma_I$, since $\hat\pi_2(\tau) = T_{\tau|_{-1}^\infty} \circ T_{\tau|_{-2}^\infty} \circ \ldots$, for any $\tau \in \Sigma_I$. Similarly it follows that  all the other conditions in Definition \ref{confSmale} are satisfied for the endomorphism $T$ defined by the maps $T_\om, \ \om \in \Sigma_I^+$. 
 
 
\  \textbf{2)} Let $Y = \overline{B(1/2, 1/2)}$ and for every $\om \in \Sigma_I^+$, define $$T_\om(z) = \frac{1}{z^2 + 2 \tilde\pi(\om)}, z \in Y.$$ Then since $Re(z^2) \ge -\frac 14, z \in Y$, and since the map  $z \to z^2$ is injective on $Y$, it follows as above that $T_\om: Y \to Y$ is well-defined, holomorphic and an injective contraction on $Y$. In this case $|z^2| \le \frac 14, z \in Y$, thus if $a, b \in I = \N^* \times \N^*$ and $a \ne b$, then it follows that $$T_{a\om}(B(1/2, 1/2)) \cap T_{b\om}(B(1/2, 1/2)) = \emptyset.$$ Also from the expression of $T_\om$, we obtain that there exist constants $C_1, C_2 \in (0, 1) $ such that $C_1 \le |T_\om'(z)| \le C_2, z \in Y, \om \in \Sigma_I^+$, and that the Bounded Distortion Property is satisfied. Moreover all the other conditions in Definition \ref{confSmale} are satisfied, so the endomorphism $T$ defined by $T_\om, \om \in \Sigma_I^+$ is a Smale conformal skew-product.

 \  \textbf{3)} Let $Y = \overline{B(0, 1)}$ and $f: Y \to Y$ be an injective  conformal contraction. Let $E$ be a countable alphabet, and a sequence of small radii $r_e \in (0, 1/3)$ and a sequence of points $z_e \in B(0, 1)$. Define $T_e(y) = r_e f(y) + z_e$, for $e \in I$. Assume that $r_e, z_e, e \in E$ are chosen such that the images of the maps $T_e$ are mutually disjoint and have accumulation points also in the interior of $Y$. For $\om \in \Sigma_I^+$, let $T_\om = T_{\om_0}.$
 Then the conditions in Definition \ref{confSmale} are satisfied for $T$. 
 
 
 \ \textbf{4)} Other examples of Smale skew-product endomorphisms can be found in \cite{MU-ETDS}.
 
 $\hfill\square$
 
 \

  Now, in general for an arbitrary Smale skew-product endomorphism $T$ and for an arbitrary sequence $\tau\in \Sigma_E$, denote the following composition of maps by,
$$
T_\tau^n
:=T_{\tau|_{-n}^{\infty}}^n
:=T_{\tau|_{-1}^{\infty}}\circ T_{\tau|_{-2}^{\infty}}\circ\ld
  \circ T_{\tau|_{-n}^{\infty}}:Y \to Y.
$$
Then  the sets $\(T_\tau^n\(Y)\)_{n=0}
^\infty$ form a  descending sequence, and 
$
\diam\(\ov{T_\tau^n(Y)})\le \l^{-n}\diam(Y).
$
But $(Y,d)$ is complete, so $
\mathop{\bi}\limits_{n=1}^\infty \ov{T_\tau^n (Y)}
$
is a point denoted by $\hat\pi_2(\tau)$. This defines the map
\begin{equation}\label{pi2hat}
\hat\pi_2:\Sigma_E\lra Y
\end{equation}
Define also the map $\hat\pi: \Sigma_E\to \Sigma_E^+\times Y$ by 
\beq\label{5111705p141}
\hat\pi(\tau)=\(\tau|_0^{\infty},\hat\pi_2(\tau)\),
\eeq
and the truncation to non-negative indices by 
$$\pi_0: \Sigma_E \lra \Sigma_E^+, \  \  \pi_0(\tau) = \tau|_0^\infty.$$
Now for arbitrary $\om\in \Sigma_E^+$ denote the $\hat \pi_2$-projection in $Y$ of the cylinder $[\om]\subset \Sigma_E$, by
\begin{equation}\label{fiberJ}
J_{T, \om}:=\hat\pi_2([\om]),
\end{equation}
and call these fractal sets the \textit{stable fibers} of  $T$.


\begin{dfn}\label{E}
Denote by $\mathcal {E}$ the collection of endomorphisms $F_T$, defined by (\ref{generalF}) for all Smale conformal skew-product endomorphisms $T$ from Definition \ref{confSmale}.
\end{dfn}

\

If $F_T$ is associated to a conformal Smale skew-product $T$, then for every $\om \in \Sigma_I^+$, let $J_{T, \om}$ be the associated \textbf{fiber limit set} of $F_T$ given by (\ref{fiberJ}). Namely,
\begin{equation}\label{JT}
J_{T, \omega}:= \big\{  \mathop{\bigcap}\limits_{n\ge 1} \overline{T_{\tau|_n^\infty}(Y)}, \  \tau \in [\omega] \big\}.
 \end{equation}
 
 \
 
The \textbf{global basic set} $J_T$ of $F_T$ is defined by, 
\begin{equation}\label{globalbs}
J_T:= \mathop{\bigcup}\limits_{\om \in \Sigma_I^+} \{\tilde \pi(\om)\}\times J_{T, \om} = \{(z, w), z \in (1, \infty)^2, w \in J_\om, \text{if} \ z = \tilde \pi(\om), \om \in \Sigma_I^+\}  \subset \C^2.
\end{equation}
 This fractal global set $J_T$ is non-compact.
 
 \
 
 \textbf{Remark.}
 
 For Example 1) above, with $T_\om(z) = \frac{1}{\bar z+\tilde\pi(\om)}, z \in Y, \om \in \Sigma_I^+$,  the points in the global basic set $J_T\subset \C^2$ are represented as a new type of  multi-dimensional continued fractions, namely
$$
(z, w) = (\tilde\pi(\om), T_{\eta|_{-1}^\infty} \circ T_{\eta|_{-2}^\infty} \circ \ldots ) 
 =\Big(m_0 + \frac{1}{m_1 + \frac{1}{\ldots}} + i(n_0+\frac{1}{n_1+\frac{1}{\ldots}}), \ T_{\eta|_{-1}^\infty} \circ T_{\eta|_{-2}^\infty} \circ \ldots  \Big),
$$
where $\eta \in \Sigma_I, \eta|_0^\infty = \om$.
 
 Also for Example 2) above, with $T_\om(z) = \frac{1}{z^2+2\tilde\pi(\om)}$ we obtain another type of multi-dimensional continued fractions.
 
 $\hfill\square$

The exact dimensionality of the projections of conditional measures on the fibers of $F_T\in\mathcal E$ follows from \cite{MU-ETDS} and gives the following:

\begin{thm}\label{exactT}
Consider an  endomorphism $F_T$ associated to an arbitrary conformal Smale skew-product $T$ by formula (\ref{generalF}). Let $\psi:\Sigma_I \to \mathbb R$ be a H\"older continuous summable potential, which has an  equilibrium measure $\mu_\psi$ on $\Sigma_I^+$.  

Then, for $\pi_{0*}\mu_\psi$-a.e $\omega\in \Sigma_I^+$, the projection $\hat\pi_{2*}\bar\mu_\psi^\omega$ of the conditional measure $\bar\mu_\psi^\omega$  of $\mu_\psi$ on $[\om]$, is exact dimensional on the limit set $J_{T, \om} \subset \C$. Moreover,  its dimension satisfies:
$$HD(\hat\pi_{2*}\bar\mu_\psi^\omega) = \frac{h_{\mu_\psi}(\sigma)}{\chi_{\mu_\psi}}.$$
\end{thm}

\

\section{General global exact dimensionality and dimension formulas.} 

This Section contains some of the main new results and methods of the paper for the non-conformal maps with singularities $F_T \in \mathcal E$. 
For an arbitrary  endomorphism $F_T \in \mathcal E$ (see Definition \ref{E}) we prove that the \textbf{global} projections $\nu$ of equilibrium measures, are exact dimensional, and we find their dimension.
Consider  the spaces $$X = (1, \infty) \times (1, \infty) \subset \C,  \ \text{and} \ \ Y  \subset \mathbb C,$$ and recall that for $x\in \mathbb R$, $\{x\}:= x - [x]$. Given $T$, let the skew-product  $F_T: X \times Y \to X \times Y$, 
\begin{equation}\label{modcx}
F_T(z, w) = \Big(\frac{1}{\{Re(z)\}} +  \frac{i}{\{Im(z)\}}, \ T_\om(w) \Big), \ z \in X,  Re(z)\notin \Q, Im(z) \notin \Q, \ z = \tilde \pi(\om),  \ w \in Y, 
\end{equation}
and $$F(z, w) = (2, 2, \frac 12), \ \text{if} \ Re(z) \in \Q \  \text{or} \ Im(z) \in \Q.$$ 

Using $\tilde \pi$ from (\ref{pitilde}), define the coding map \ $\tilde \pi_Y: \Sigma_I^+ \times Y \to X \times Y, $
\begin{equation}\label{pitildeY}
\tilde \pi_Y(\omega, y) = (\tilde\pi(\omega), y), \ (\omega, y) \in \Sigma_I^+ \times Y.
\end{equation}
Recall also that we denoted the truncation to non-negative indices by,
$$\pi_0: \Sigma_I \to \Sigma_I^+, \ \pi_0(\eta) = \eta|_0^\infty, \ \eta \in \Sigma_I, $$
and given the conformal Smale skew-product $T$, there exists from (\ref{5111705p141}) a coding map, 
\begin{equation}\label{pihat}
\hat\pi : \Sigma_I \to \Sigma_I^+ \times Y, \ \ \hat \pi(\eta) = (\pi_0(\eta), \hat\pi_2(\eta)), \ \eta \in \Sigma_I.
\end{equation}

Introduce also the projection 
\begin{equation}\label{piclar}
\pi: \Sigma_I \to X \times Y, \ \pi:= \tilde \pi_Y \circ \hat \pi, \ \pi(\eta) = (\tilde \pi(\eta|_0^\infty), \hat\pi_2(\eta)), \eta \in \Sigma_I.
\end{equation}

As $I = \N^*\times \N^*$, denote the canonical projections on ``coordinates'' of points in $\Sigma_I^+$ by, $$p_1: \Sigma_I^+ \to \Sigma_{\mathbb N^*}^+, \ \text{and} \ p_2:\Sigma_I^+ \to \Sigma_{\mathbb N^*}^+,$$ where for any $((m_0, n_0), (m_1, n_1), \ldots) \in \Sigma_I^+$, with $m_i, n_i \in \N^*,  \ i \in \N$, 
\begin{equation}\label{p1}
 p_1((m_0, n_0), (m_1, n_1), \ldots) = (m_0, m_1, \ldots),  \ \  p_2((m_0, n_0), (m_1, n_1), \ldots) = (n_0, n_1, \ldots).
\end{equation}

If $\mu^+$ is a $\sigma$-invariant measure on $\Sigma_I^+$, denote the projection measures of $\mu^+$ on $\Sigma_{\N^*}^+$ by 
\begin{equation}\label{mu1}
\mu_1:= p_{1*} \mu^+, \ \text{and} \ \mu_2:= p_{2*} \mu^+,
\end{equation}
and call $\mu_1, \mu_2$ the \textbf{marginal mesures} of $\mu^+$. The entropies of $\mu_1, \mu_2$ with respect to the shift on $\Sigma_{\N^*}^+$ are called the \textbf{marginal entropies} of $\mu^+$.

\

Since we work with continued fractions, define for every $n \in \N^*$ the contraction map,
\begin{equation}\label{varphi}
\varphi_n(x) = \frac{1}{x+n}, \ x \in [0, 1).
\end{equation}
Denote the coding $\rho_0: \Sigma_{\N^*}^+ \to (0, 1)$, as the representation in continued fractions, 
\begin{equation}\label{ro0}
 \rho_0(\omega) = \frac{1}{\omega_0 + \frac{1}{\omega_1 + \ldots}}, \ \text{for} \ \omega=(\omega_0, \omega_1, \ldots) \in \Sigma_{\N^*}^+.
 \end{equation}
For an arbitrary $\sigma$-invariant measure $\mu$ on $\Sigma_I$, denote by $\mu^+:= \pi_{0*}\mu$ on $\Sigma_I^+$, and let $\mu_1, \mu_2$ be the measures from (\ref{mu1}) associated to $\mu^+$.

Using (\ref{mu1}) denote the \textbf{marginal Lyapunov exponents of $\mu^+$} by:
\begin{equation}\label{lyap}
\chi_1(\mu^+) := -\int_{\Sigma_{\N^*}^+} \log |\varphi'_{\zeta_0}(\rho_0\sigma\zeta)| \ d\mu_1(\zeta).
\end{equation}
$$
\chi_2(\mu^+) := -\int_{\Sigma_{\N^*}^+} \log |\varphi'_{\zeta_0}(\rho_0\sigma\zeta)| \ d\mu_2(\zeta).
$$
If $\mu^+ = \pi_{0*}\mu$, write also $\chi_1(\mu), \chi_2(\mu)$ for $\chi_1(\mu^+), \chi_2(\mu^+)$ respectively.

\

Define the \textbf{Lyapunov exponent of the endomorphism} $T$ with respect to a shift-invariant probability measure $\mu$ on $\Sigma_I$ by, 
\begin{equation}\label{lyapglob}
\chi_T(\mu):= -\int_{\Sigma_I} \log |T'_{\eta|_0^\infty}(\hat\pi_2(\eta))| \ d\mu(\eta).
\end{equation}
If $T$ is fixed, denote also $\chi(\mu)$. Notice that $\chi_1(\mu^+), \chi_2(\mu^+), \chi(\mu)$ are all positive. Also let
\begin{equation}\label{lambda1}
\lambda_1(\mu^+):= \exp(-\chi_1(\mu^+)), \ \lambda_2(\mu^+):= \exp(-\chi_2(\mu^+))
\end{equation}
So $\lambda_1(\mu^+), \lambda_2(\mu^+) \in (0, 1)$. If $\mu^+ = \pi_{0*}\mu$,  write also $\lambda_1(\mu), \lambda_2(\mu)$ for $\lambda_1(\mu^+), \lambda_2(\mu^+)$.

\

If the endomorphism $F_T$ is associated to a conformal Smale skew-product $T$, then for any $\om \in \Sigma_I^+$ the fiber limit set was denoted by $J_{T, \om}$. Let $J_T$ be the \textbf{global basic set} of $F_T$, 
$$
J_T:= \mathop{\bigcup}\limits_{\om \in \Sigma_I^+} \{\tilde \pi(\om)\}\times J_{T, \om} \ \subset \C^2.
$$
 This global fractal basic set $J_T$ is non-compact, since  $X = (1, \infty) \times (1, \infty)$.

Recall that $I = \N^* \times \N^*$, and from (\ref{piclar}), $$\pi: \Sigma_I \longrightarrow (1, \infty)^2 \times Y, \  \ \pi(\eta) = (\tilde \pi(\eta|_0^\infty), \hat\pi_2(\eta)), \eta \in \Sigma_I.$$
Thus if $\eta = (\ldots, (m_{-1}, n_{-1}), (m_0, n_0), (m_1, n_1), \ldots)$ with $m_j, n_j \in \N^*, j \in \Z$, then 
\begin{equation}\label{pieta}
\pi(\eta) = \big(m_0+\frac{1}{m_1+\frac{1}{m_2 + \ldots}} + i (n_0 + \frac{1}{n_1+ \frac{1}{n_2 + \ldots}}), \ T_{\eta|_{-1}^\infty}\circ T_{\eta|_{-2}^\infty}\circ \ldots\big).
\end{equation}

\

  Recall now the notations of Section 2, in particular the $F_T$-invariant measure $\nu_\psi = \pi_*\mu_\psi$ on $J_T$, where $\mu_\psi$ is an equilibrium measure on $\Sigma_I$. 
  
  \
  
  The next Theorem proves the \textbf{global exact dimensionality} of the measure $\nu_\psi$ on the global basic set $J_T$, and gives also the  \textbf{dimension formula} for $HD(\nu_\psi)$, in terms of Lyapunov exponents and marginal entropies.
  
  \
  
\begin{thm}\label{modexT}
Let the endomorphism $F_T: X \times Y \to X \times Y$ be associated to a conformal Smale skew-product $T$ by (\ref{generalF}). 
Let $\psi: \Sigma_I \to \mathbb R$ be a H\"older continuous summable potential with  $\mu_\psi$ its equilibrium measure on $\Sigma_I$, and $\pi$ the coding from (\ref{pieta}), and $\pi_0:\Sigma_I \to \Sigma_I^+$  the canonical truncation to positive coordinates. Denote the measures $$\mu^+_\psi:= (\pi_0)_* \mu_\psi \ \text{on} \ \Sigma_I^+, \ \text{and} \  \nu_\psi:= \pi_* \mu_\psi \ \text{on} \  X \times Y.$$ Let $\mu_{1, \psi}, \mu_{2, \psi}$ be the marginal  measures $\mu_1, \mu_2$ from (\ref{mu1}) associated to $\mu_\psi^+$.  
Then,

a) The projection measure  $\nu_\psi$ is exact dimensional on the global basic set $J_T \subset \C^2$.

b) If $\lambda_1(\mu_\psi) < \lambda_2(\mu_\psi)$,  the  pointwise (and Hausdorff) dimension of $\nu_\psi$ is,
$$\delta(\nu_\psi) = HD(\nu_\psi) = \frac{h_{\mu_\psi} - h_{\mu_{1, \psi}}(1-\frac{\chi_2(\mu_\psi)}{\chi_1(\mu_\psi)})}{\chi_2(\mu_\psi)} + \frac{h_{\mu_\psi}}{\chi(\mu_\psi)}.$$

c) If $\lambda_1(\mu_\psi) \ge \lambda_2(\mu_\psi)$, then 
$$\delta(\nu_\psi) = HD(\nu_\psi) = \frac{h_{\mu_\psi} - h_{\mu_{2, \psi}}(1-\frac{\chi_1(\mu_\psi)}{\chi_2(\mu_\psi)})}{\chi_1(\mu_\psi)} + \frac{h_{\mu_\psi}}{\chi(\mu_\psi)}.$$

\end{thm}

\begin{proof}

Firstly let take the projections of the conditional measures of the measure $\mu$ onto the fibers $J_{T, \omega}, \omega \in\Sigma_I^+$ defined in (\ref{fiberJ}), then look at the projections of $\nu_\psi$ on the first complex coordinate $z$. However, the first complex coordinate map $F_{T, 1}(z, w)$, of $F_T(z, w)$, is only piecewise differentiable on countably many pieces and it is not conformal on these pieces, and the measure $\mu_\psi$ can have different Lyapunov exponents $\chi_1(\mu_\psi), \chi_2(\mu_\psi)$ in the two real directions of the $z$-plane.
Moreover, in the first real coordinate of the $z$-plane we have to be careful since our iterated function system $(\varphi_n)_n$ is not uniformly contracting. All these facts require new ideas and methods. We will prove a Volume Lemma in the $z$-direction, and then a more difficult Global Volume Lemma for $\nu_\psi$ itself. This will imply the exact dimensionality for  $\nu_\psi$ on $X\times Y$ and will give the formula for the dimension of $\nu_\psi$.   

The proof contains several \textbf{main steps}, which are  detailed below:

\

\ \ \textbf{Step 1. Codings, notations, and the measure $\Ups$.}

It will be more convenient to work on $(0, 1) \times (0, 1)$ instead of $X$, so denote $Z := (0, 1) \times (0, 1),$
and consider the bijective transformation $$\theta: X \to Z, \ \theta(x, y) = (\frac 1x, \frac 1y), \ (x, y) \in X$$
Recalling the notation in  (\ref{ro0}), introduce then the coding map $\rho: \Sigma_I^+ \to Z$, $$\rho((m_0, n_0), (m_1, n_1), \ldots) = (\frac{1}{m_0 + \frac{1}{m_1+\ldots}}, \frac{1}{n_0+\frac{1}{n_1+\ldots}}) = (\rho_0(m_0, m_1, \ldots), \rho_0(n_0, n_1, \ldots)),$$ for every $ ((m_0, n_0), (m_1, n_1), \ldots) \in \Sigma_I^+$.
Let us denote also $$\rho_Y: \Sigma_I^+ \times Y \to Z \times Y, \ \rho_Y(\omega, y) = (\rho(\omega), y)$$
For the H\"older continuous summable potential $\psi: \Sigma_I \to \R$ from above, define the measures: 
\begin{equation}\label{ups}
\vartheta_\psi:= (\rho_Y)_* \mu^+_\psi \ \text{on} \ Z\times Y, \ \text{and} \  \ \Ups_\psi:= \rho_*\mu^+_\psi \ \text{on} \ Z
\end{equation}
Notice that $(\theta\times id_Y)_*\nu_\psi = \vartheta_\psi$.\ 
Let assume the potential $\psi$ \textbf{is fixed}, and drop the index $\psi$ from the notations above, so write $\mu$ for $\mu_\psi$, $\nu$ for $\nu_\psi$, $\mu^+$ for $\mu_\psi^+$. Also, $\Ups$ for $\Ups_\psi$, $\mu_1$ for $\mu_{1, \psi}$, $\mu_2$ for $\mu_{2, \psi}$, and $\vartheta$ for $\vartheta_\psi$. 
Recalling (\ref{mu1}) denote by $\mu_1$, $\mu_2$ respectively the measures $\mu_{1, \psi}$ and $\mu_{2, \psi}$ (associated to $\mu_\psi^+$). 
From (\ref{lyap}), (\ref{lyapglob}) write $\chi_1$ for $\chi_1(\mu_\psi^+)$, $\chi_2$ for $\chi_2(\mu_\psi^+)$, $\chi$ for $\chi(\mu_\psi)$. By (\ref{lambda1}) denote $\lambda_1$ for $\lambda_1(\mu_\psi^+)$ and $\lambda_2$ for $\lambda_2(\mu_\psi^+)$.

\

\ \ \textbf{Step 2. Past-independent potentials, and the projection measure $\Ups_1$.}

We observe now that as in \cite{Bo}, there exists a past-independent H\"older continuous summable function $\psi^+$ on $\Sigma_I$, which is cohomologous to $\psi$ in the class of bounded H\"older continuous functions, where by \textit{past-independent} we mean that $$\psi^+(\eta) = \psi^+(\eta'),$$ for any $\eta, \eta' \in \Sigma_I$ with $\eta_j = \eta'_j$, for $ j \ge 0$. 
Denote the restriction of $\psi^+$ to $\Sigma_I^+$ by $\tilde\psi^+$.
Let us now look closer at these two probability measures $\mu$ on $\Sigma_I$ and $\mu^+$ on $\Sigma_I^+$. Since  $\mu = \mu_\psi = \mu_{\psi^+}$, it follows from Gibbs property that for any integer $n >1 $ and any  $\zeta \in \Sigma_I$, 
 $$\mu([\zeta_{0} \ldots \zeta_n]) \approx \exp(S_n \psi^+(\zeta)) - n P(\psi^+)),$$
where the pressure $P(\psi^+)$ is taken with respect to the shift on $\Sigma_I$, and where the comparability constants do not depend on $n, \zeta$.
But, on the other hand, if $\mu_{\tilde \psi^+}$ denotes the equilibrium measure of $\tilde \psi^+$ on $\Sigma_I^+$, we have from the Gibbs property that, for any $\xi \in \Sigma_I^+$, any $n >1$ and any $\zeta\in \Sigma_I$ with $\zeta|_0^\infty = \xi$, the following estimate holds:
$$\mu_{\tilde \psi^+}([\xi_0 \ldots \xi_n]) \approx \exp(S_n\tilde\psi^+(\xi) - n P(\tilde\psi^+)) = \exp(S_n\psi^+(\zeta) - nP(\tilde\psi^+)),$$
where now the pressure $P(\tilde\psi^+)$ is taken with respect to the shift on $\Sigma_I^+$, and where the comparability constants do not depend on $n, \xi$. But clearly $P(\psi^+) = P(\tilde\psi^+)$. Hence from the uniqueness of Gibbs measures for given H\"older continuous summable potentials on $\Sigma_I^+$, 
\begin{equation}\label{muplus}
\pi_{0*}\mu_\psi = \mu_\psi^+ = \pi_{0*} \mu_{\psi^+} = \mu_{\tilde\psi^+}
\end{equation}

Also, since $(\Sigma_I, \sigma, \mu)$ is the natural extension of the system $(\Sigma_I^+, \sigma, \mu^+)$ (or by using the Brin-Katok formula and the estimates for the measure on Bowen balls)),  it follows that 
\begin{equation}\label{entplus}
h_{\mu} = h_{\mu^+}
\end{equation}

Now, consider the canonical projection on the $X$-coordinate $$p_X: X \times Y \to X$$
Since by (\ref{ups}), $\Ups = (\theta\circ p_X)_*\nu = \rho_*\mu^+$ on $Z$, we see immediately from the definition of the pointwise dimension $\delta$ that, if $z = \theta(x, y)$ with $(x, y) \in X$, then
\begin{equation}\label{hdeg}
\delta(\Ups)(z) = \delta((p_X)_*\nu)(x, y) 
\end{equation}

Recall that $\Ups$ is a measure on $Z = (0, 1) \times (0, 1)$, and denote its canonical projection in the first coordinate on $(0, 1)$ by $\Ups_1$.
Hence, in the notations of (\ref{p1}), (\ref{mu1}) and (\ref{ro0}),
\begin{equation}\label{upsmu1}
\Ups_1 = (\rho_0)_*\mu_1
\end{equation}
As $\mu^+$ is $\sigma$-invariant on $\Sigma_I^+$, then $\mu_1$ is $\sigma$-invariant on $\Sigma_{\N^*}^+$. 

\

\ \ \textbf{Step 3. Pointwise dimension for $\Ups_1$.}

From the definition (\ref{ro0}) of the map $\rho_0$ as the representation of irrational numbers as continued fractions, we obtain that the measure $\Ups_1$ is the projection on $(0, 1)$ of an invariant probability measure $\mu_1$ on $\Sigma_{\N^*}^+$, with respect to the countable iterated function system of (\ref{varphi}) in the first real coordinate. 

However,  the map $\varphi_1$ has a parabolic point at $0$. Thus $(\varphi_n)_n$ is not a uniformly contracting iterated function system. We will associate to it a uniformly contracting iterated function system.
Indeed let a point $x \in (0, 1), x \ne \frac{\sqrt 5 - 1}{2} = \frac{1}{1+\frac{1}{\ldots}}$. Define the iterated function system $\{\hat\varphi_n, n \ge 1\} = \{\varphi_1^k\circ\varphi_j,  \ \varphi_j \circ \varphi_1^k, \ k \ge 0, j >1\}$, and notice that this sytem is uniformly contracting on $[0, 1)$.  Let us denote the projection of $\mu_1$  on $[0, 1)$  relative to the system $(\hat\varphi_n)_n$ by $\hat\Ups_1$. Then since $x \ne \varphi_1(\varphi_1(\ldots))$, it follows that for any small $r>0$ there exists an integer $n(r)\ge 1$ and $j_1, \ldots, j_{n(r)} \in \N^*$, such that 
\begin{equation}\label{ups1hat}
\hat\Ups_1(B(x, r)) = \Ups_1(B(x, r)) = \mu_1([j_1\ldots j_{n(r)}]).
\end{equation}
Thus since the new system $(\hat\varphi_n)_n$ is uniformly contracting, it follows by the main result in \cite{MU-Adv} (restricted to the case when the parameter space  consists of only one point), that the measure $\hat\Ups_1$ is exact dimensional on $[0, 1)$. Hence from (\ref{ups1hat}), $\Ups_1$ is also exact dimensional on $[0, 1)$, and the pointwise dimension of $\Ups_1$  is equal to the Hausdorff dimension of $\Ups_1$. 

Let us now denote the pointwise dimension of $\Ups_1$ by $\delta_1$. In our case there are no overlaps in the countable iterated function system, so the open set condition  is satisfied. Thus, the projectional entropy of $\mu_1$ is the same as its usual entropy $h_{\mu_1}$. Hence from \cite{MU-Adv} it follows that for $\Ups_1$-a.e. point $x \in (0, 1)$, 
\begin{equation}\label{delta1}
\delta_1 = \delta(\Ups_1)(x) = \mathop{\lim}\limits_{r\to 0} \frac{\log \Ups_1(B(x, r))}{\log r} = \frac{h_{\mu_1}}{\chi_1}.
\end{equation}

\ \ \textbf{Step 4. Geometry in the $z$-direction, and generic points.}

Let us assume first, as in part b) of the statement, that 
\begin{equation}\label{lambda12}
\lambda_1 < \lambda_2
\end{equation}
 Consider arbitrary numbers $n >1$ and $\vp>0$, and define the Borel  set $C(n, \vp) \subset [0, 1)$ by:
$$
C(n, \vp):= \{ x\in (0, 1), \ \frac{\log \Ups_1(B(x, r))}{\log r} \in (\delta_1 - \vp, \delta_1 + \vp), \text{for} \ 0 < r \le \lambda_2^{n(1-\vp)}\}
$$
From the exact dimensionality of  $\Ups_1$, it follows that for every $\vp>0$ there exists an integer $n(\vp) >1$ and a positive function $\kappa(\cdot)$ with $\mathop{\lim}\limits_{\vp \to 0} \kappa(\vp) = 0$, such that for every $n > n(\vp)$,
\begin{equation}\label{cne}
\Ups_1(C(n, \vp)) > 1- \kappa(\vp)
\end{equation}

We identify in the sequel the space $\Sigma_I^+$ with $\Sigma_{\N^*}^+\times \Sigma_{\N^*}^+$ by the map $\Psi: \Sigma_I^+\to \Sigma_{\N^*}^+\times \Sigma_{\N^*}^+, $ 
\begin{equation}\label{psio}
\Psi(\zeta) = (\omega, \eta),
\end{equation}
where for any $\zeta = ((m_0, n_0), (m_1, n_1), \ldots) \in \Sigma_I^+$,  write $\omega = (m_0, m_1, \ldots) \in \Sigma_{\N^*}^+$, and $\eta = (n_0, n_1, \ldots)\in \Sigma_{\N^*}^+$.
Given the contractions $\phi_n, n \ge 1$  of (\ref{varphi}), and any $\omega \in \Sigma_{\N^*}^+$, denote $$\phi^n_\omega := \phi_{\omega_0}\circ\phi_{\omega_1} \circ \ldots \circ \phi_{n-1}$$

Now for arbitrary $n > 1$, $\vp>0$, define the following Borel measurable set 
\begin{equation}\label{aneps}
\begin{aligned}
A(n, \vp) &:= \{(\omega, \eta)  \in \Sigma_{\N^*}^+\times \Sigma_{\N^*}^+, \ |\frac{S_{j}\psi^+(\omega, \eta)}{j}-\int\psi^+ d\mu^+| < \vp, \ \text{and} \ \\
&|(\phi_{\omega}^{j})'(\rho_0\sigma^j\omega)| \in (e^{j(-\chi_1-\vp)}, e^{j(-\chi_1 + \vp)}) \ \text{and} \ |(\phi_\eta^j)' (\rho_0\sigma^j\eta)| \in (e^{(j(-\chi_2 - \vp)}, e^{j(-\chi_2 + \vp)}), \ \forall  j \ge n\}
\end{aligned}
\end{equation}
Since $\mu^+$ and $\mu_1, \mu_2$ are ergodic with respect to the shift maps, and using  Birkhoff Ergodic Theorem and (\ref{lyap}), it follows that $\mu^+(A(n, \vp))$ is close to 1, for large $n$. Without loss of generality assume that for the functions $n(\vp)$ and $\kappa(\vp)$ above, we have for any $n > n(\vp)$,
\begin{equation}\label{ane}
\mu^+(A(n, \vp)) > 1-\kappa(\vp).
\end{equation}
 Therefore  if $n > n(\vp)$, then
$\mu_1(p_1(A(n, \vp))) > 1-\kappa(\vp)$. Denote  by $$\tilde C(n, \vp):= C(n, \vp) \cap \rho_0(p_1(A(n, \vp))) \ \subset (0, 1)$$
Since by (\ref{upsmu1}), $\Ups_1 = \rho_{0*}\mu_1$, it follows from (\ref{cne}) and (\ref{ane}) that if $n > n(\vp)$,
\begin{equation}\label{tildec}
\Ups_1(\tilde C(n, \vp)) > 1-2\kappa(\vp)
\end{equation}

Now consider $(\omega, \eta) \in A(n, \vp)$ and its $\rho$-projection, $z = \rho(\omega, \eta) =  (\rho_0(\omega), \rho_0(\eta)) \in (0, 1) \times (0, 1)$. We will estimate the measure $\Ups(B(z, \lambda_2^{(1+\vp)n}))$ by covering a  large portion of $B(z, \lambda_2^{(1+\vp)n})$ with an optimal cover, consisting of $\rho$-projections of cylinders of type $$[\omega_0'\ldots \omega_n'] \times [\eta_0' \ldots \eta_n']$$ The $\Ups$-measure of the projection of such cylinder is equal to the $\mu^+$-measure of  the cylinder $[(\omega_0', \eta_0') \ldots (\omega_n', \eta_n')]$, and it can be estimated using the Gibbs property of $\mu^+$. So we shall estimate the number of such cylinders in the optimal cover, by looking at their projections on the first real coordinate. 
From (\ref{psio}) we identify $\Sigma_I^+$ with $\Sigma_{\N^*}^+\times \Sigma_{\N^*}^+$, and can consider that $\psi^+$ is defined on $\Sigma_{\N^*}^+ \times \Sigma_{\N^*}^+$. 
As the IFS $\{\varphi_n, n \ge 1\}$ satisfies Open Set Condition, and by using the Gibbs property of $\mu^+$, it follows that there exists a constant $C>0$ such that for any $n \ge 1$ and any $(\omega', \eta') \in \Sigma_I^+$,  we have:
\begin{equation}\label{projmu}
\begin{aligned}
 \Ups(\rho([(\omega_0', \eta_0') \ldots (\omega_n', \eta_n')])) = &\mu^+[(\omega_0', \eta_0') \ldots (\omega_n', \eta_n')] \in \\
 &\in \Big(\frac{e^{S_n\psi^+(\omega', \eta') - nP(\psi^+)}}{C}, \ C e^{S_n\psi^+(\omega', \eta') - nP(\psi^+)}\Big)
 \end{aligned}
\end{equation}

Introduce now for arbitrary $n >1$, $\vp>0$, the following Borel subset of $(0, 1)\times (0, 1)$, 
\begin{equation}\label{om}
\hat\Om(n, \vp):= \big\{z \in Z, \ \   \Ups\big(B(z, r) \ \cap \ \rho(A(n, \vp)) \cap (C(n, \vp)\times (0, 1)) \big) > \frac 12 \Ups(B(z, r)), \ \forall \ 0 < r \le \lambda_2^{(1+\vp)n} \big\}
\end{equation}
Denote the complement of $\hat\Omega(n, \vp)$ in $(0, 1) \times (0, 1)$ by $\hat\Omega^c(n, \vp))$. From definition, if $z \in \hat\Omega^c(n, \vp)$, then there  exists $ 0 < r = r(z) < \lambda_2^{(1+\vp)n}$ so that: 
\begin{equation}\label{doi}
\Ups(B(z, r) \setminus (\rho(A(n, \vp)) \cap C(n, \vp) \times (0, 1))) >  \frac 12 \Ups(B(z, r))
\end{equation}
In this case we can cover the set $\hat\Omega^c(n, \vp)$ with balls of type $B(z, r(z))$ for all $ z \in \hat\Omega^c(n, \vp)$. Then, from Besicovitch Covering Theorem applied to a bounded subset of $Z = (0, 1) \times (0, 1)$, there exists a subcover with such balls for $z \in \mathcal F$, such that $$\hat\Omega^c(n, \vp) \subset \mathop{\cup}\limits_{z \in \mathcal F} B(z, r(z)), $$
and the multiplicity of this subcover is finite and bounded  by a constant $M$ independent of $n, \vp$. 
Thus from  (\ref{doi}) and the bounded multiplicity of $\mathcal F$, and by using (\ref{cne}) and (\ref{ane}), one obtains:
$$
\begin{aligned}
\Ups(\hat\Omega^c(n, \vp)) &\le \mathop{\sum}\limits_{z\in \mathcal F} \Ups(B(z, r(z)) \le 2\mathop{\sum}\limits_{z \in \mathcal F} \Ups(B(z, r) \setminus (\rho(A(n, \vp)) \cap C(n, \vp) \times (0, 1))) \\
& \le 2M \Ups\big(\mathop\cup \limits_{ z \in \mathcal F} B(z, r(z)) \setminus (\rho(A(n, \vp)) \cap C(n, \vp) \times (0, 1)) \big) \\ &\le 2M \Ups\big(Z\setminus (\rho(A(n, \vp)) \cap C(n, \vp) \times (0, 1)) \big) \\ & =  2M\big(1 - \Ups(\rho(A(n, \vp)) \cap C(n, \vp) \times (0, 1))\big) < 4M\kappa(\vp)
\end{aligned}
$$
Hence, $
\Ups(\hat\Omega(n, \vp)) \ge 1 - 4M\kappa(\vp)$. \ 

Now we can do the same argument as before for the projection measure $\Ups_1$ of $\Ups$ on the first coordinate, i.e. $\Ups_1 = p_{1*} \Ups$, where $p_1$ denotes the projection on the first real coordinate. Hence we obtain a Borel set $\Omega_1(n, \vp) \subset (0, 1)$, with $$\Ups_1(\Om_1(n, \vp)) > 1- 4M \kappa(\vp),$$ and such that for any point $x \in \Om_1(n, \vp)$ and any $0 < r < \lambda_2^{n(1+\vp)}$, we have:
\begin{equation}\label{om1}
\Ups_1\big(B(x, r) \cap C(n, \vp) \cap p_1(\rho(A(n, \vp))) \big) > \frac 12 \Ups_1(B(x, r)).
\end{equation}

Denote in the sequel
\begin{equation}\label{omn}
\Om(n, \vp) := \hat\Om(n, \vp) \cap p_1^{-1}\Om_1(n, \vp)
\end{equation}
From above, since $\Ups(p_1^{-1}\Om_1(n, \vp)) = \Ups_1(\Om_1(n, \vp))$, it follows that 
\begin{equation}\label{omups}
\Ups(\Om(n, \vp)) > 1-8M \kappa(\vp).
\end{equation}

Let us take a point $z \in \Omega(n, \vp)$ and assume $z = \rho(\omega, \eta)$.
Thus for any integer $n' \ge n$,
\begin{equation}\label{upse}
\Ups(B(z, \lambda_2^{(1+\vp)n'})) \le 2 \Ups\big(B(z, \lambda_2^{(1+\vp)n'}) \cap \rho(A(n, \vp)) \cap C(n, \vp) \times (0, 1))\big).
\end{equation}

So for any $n'\ge n$, in order to estimate the $\Ups$-measure of the  ball $B(z, \lambda_2^{(1+\vp)n'}) \subset (0, 1) \times (0, 1),$ it is enough to consider only its generic points in the sense of (\ref{upse}). Let us take then a point $w \in B(z, \lambda_2^{(1+\vp)n'}) \cap \rho(A(n, \vp)) \cap C(n, \vp) \times (0, 1)$, with $w = \rho(\omega', \eta')$. Then, from the definition of $\Ups$ as the projection measure  $\rho_*\mu^+$, $$\Ups(\rho([\omega_0'\ldots \omega'_{n'}]\times [\eta_0'\ldots \eta'_{n'}]) = \mu^+([\omega_0'\ldots \omega'_{n'}]\times [\eta_0'\ldots \eta'_{n'}]).$$
For any $j \ge 1$ and any $(\omega', \eta') \in \Sigma_I^+$ (recall that $\Sigma_I^+$ is identified by $\Psi$ with $\Sigma_{\N^*}^+\times \Sigma_{\N^*}^+$), denote the projection of the associated $j$-cylinder $[\omega_0' \ldots \omega_j'] \times [\eta_0' \ldots \eta_j']$ by:
\begin{equation}\label{drept}
R_j(\omega', \eta') := \rho([\omega_0' \ldots \omega_j'] \times [\eta_0' \ldots \eta_j']) \ \subset (0, 1) \times (0, 1).
\end{equation}
Now let $(\omega', \eta') \in \Sigma_I^+$ with $\rho(\omega', \eta') \in \rho(A(n, \vp) \cap C(n, \vp) \times (0, 1))$. So for any $n' \ge n$, $|\frac{S_{n'}\psi^+(\omega', \eta')}{n'} - \int \psi^+ d\mu^+| < \vp$.
Next we use the definition of $\Ups = \rho_*\mu^+$ and the Gibbs property of $\mu^+$ from (\ref{projmu}). But $(\omega', \eta')$ is generic, and $$P(\psi^+) = h_{\mu^+} + \int\psi^+ d\mu^+.$$  Hence for any $(\omega', \eta') \in A(n, \vp) \cap \rho^{-1}(C(n, \vp) \times (0, 1))$ and any $n' \ge n$ we have,
\begin{equation}\label{upsgen}
\frac 1C e^{-n'(h_{\mu^+}+\vp)} <  \Ups(R_{n'}(\omega', \eta')) < C e^{-n'(h_{\mu^+}-\vp)},
\end{equation}
where the constant $C>0$ does not depend on $n, n', \omega', \eta'$.
Recall that $z = \rho(\omega, \eta) = (x, y) \in \Om(n, \vp)$, and we want to cover the set $$B(z, \lambda_2^{n'(1+\vp)}) \cap \rho(A(n, \vp)) \cap  C(n, \vp) \times (0, 1),$$ with rectangles in $(0, 1) \times (0, 1)$ of type $R_{n'}(\omega', \eta')$, \ where $(\omega', \eta') \in A(n, \vp) \cap \rho^{-1}(C(n, \vp) \times (0, 1))$. But by (\ref{drept}), $$R_{n'}(\omega', \eta') = \rho_0([\omega'_0 \ldots \omega'_{n'}]) \times \rho_0([\eta'_0 \ldots \eta'_{n'}]).$$ Hence if $(\omega', \eta')  \in A(n, \vp) \cap \rho^{-1}(C(n, \vp) \times (0, 1))$,  and $n'\ge n$, and $(z_1', z_2'):= \rho(\omega', \eta')$, then by (\ref{aneps}),
\begin{equation}\label{rn}
B(z_1', e^{n'(\chi_1-\vp)}) \times B(z_2', e^{n'(\chi_2-\vp)})   \subset \  R_{n'}(\omega', \eta')  \subset  B(z_1', e^{n'(\chi_1+\vp)}) \times B(z_2', e^{n'(\chi_2 +\vp)}).
\end{equation}

\

\ \ \textbf{Step 5. Estimates for the number of covering rectangles and $\Ups$.}

Let us now cover the set $B(z, \lambda_2^{n'(1+\vp)}) \cap \rho(A(n, \vp)) \cap  C(n, \vp) \times (0, 1))$ with rectangles $R_{n'}(\omega', \eta')$, with $(\omega', \eta')$ from a finite family $\mathcal G$, such that the projections of these rectangles on the first real coordinate intersect with multiplicity bounded by $M$; this is possible by using the Besicovitch Covering Theorem.

Denote the number of rectangles in $\mathcal G$ by $N(n', \vp)$; clearly $N(n', \vp)$ depends also on $z$. We want to estimate this number of rectangles $N(n', \vp)$.
In order to do this, notice from (\ref{rn}) that the projection on first coordinate of an arbitrary rectangle from $\mathcal G$ is a ball of some radius $r$, with all these radii $r$ satisfying$$r \in (\lambda_1^{n'(1+\vp)}, \lambda_1^{n'(1-\vp)})$$

 Denote the projection on first coordinate of $R_{n'}(\omega', \eta')$ by $D_{n'}(\omega', \eta')$, and denote by $\mathcal G_1$ the set of projections  $D_{n'}(\omega', \eta')$ for the rectangles from $\mathcal G$. From construction, $\mathcal G_1$ has multiplicity bounded above by $M$. Since in any set  $D_{n'}(\omega', \eta')$ from $\mathcal G_1$ there are points from $C(n, \vp)$, and $\lambda_1 < \lambda_2$,  there exists a constant (denoted also $C$) so that for every $n' \ge n$,
\begin{equation}\label{estdn}
\frac 1C \lambda_1^{n'(\delta_1+\vp)} < \Ups_1(D_{n'}(\omega', \eta')) < C \lambda_1^{n'(\delta_1-\vp)}
\end{equation}
From the definition of $\mathcal G_1$ and (\ref{omn}), it follows that the sets $D_{n'}(\omega', \eta')$ from $\mathcal G_1$ cover $B(x, \lambda_2^{n'(1 +\vp)}) \cap C(n, \vp) \cap \rho_0(p_1(A(n, \vp)))$. 
But from (\ref{om1}), 
$$\Ups_1(B(x, \lambda_2^{n'(1 +\vp)}) \cap C(n, \vp) \cap \rho_0(p_1(A(n, \vp)))) \ge \frac 12 \Ups_1(B(x, \lambda_2^{n'(1+\vp)})
$$
Thus from (\ref{estdn}) and as $x \in C(n, \vp)$ and $\mathcal G_1$ has multiplicity bounded by $M$, it follows that 
$$
\frac {1}{MC} \lambda_1^{n'(\delta_1+\vp)} N(n', \vp) < \lambda_2^{n'(\delta_1-\vp)} \ \text{and} \  2C \lambda_1^{n'(\delta_1-\vp)} N(n', \vp) > \lambda_2^{n'(\delta_1 + \vp)}
$$
So there exists a constant $C_1>0$ such that for all $z\in \Om(n, \vp)$ and any $n' \ge n$,
\begin{equation}\label{Nn}
\frac {1}{C_1} \lambda_2^{n'(\delta_1+\vp)}\lambda_1^{n'(-\delta_1+\vp)} \le N(n', \vp) \le C_1 \lambda_2^{n'(\delta_1-\vp)} \lambda_1^{-n'(\delta_1+\vp)}
\end{equation}

We now use this estimate of $N(n', \vp)$, in order to estimate  the $\Ups$-measure of  $B(z, \lambda_2^{n'(1+\vp)})$. From (\ref{upse}) it follows however that it is enough to estimate the measure $$\Ups\big(B(z, \lambda_2^{(1+\vp)n'} \cap \rho(A(n, \vp)) \cap C(n, \vp) \times (0, 1))\big),$$ and notice that the set $B(z, \lambda_2^{(1+\vp)n'} \cap \rho(A(n, \vp)) \cap C(n, \vp) \times (0, 1))$ is covered with the $N(n', \vp)$ rectangles from $\mathcal G$. But the $\Ups$-measure of every rectangle from the cover $\mathcal G$  was estimated in (\ref{upsgen}). Thus there exists a constant $C_2 > 0$ so that for any $z \in \Om(n, \vp)$
and for any integer $n'>n$, 
\begin{equation}\label{gata}
\hspace{.54in} \frac{1}{C_2} \lambda_2^{n'(\delta_1+\vp)} \lambda_1^{n'(-\delta_1+\vp)} e^{-n'(h_{\mu^+}+\vp)} < \Ups(B(z, \lambda_2^{n'(1+\vp)})) < C_2\lambda_2^{n'(\delta_1-\vp)} \lambda_1^{n'(-\delta_1-\vp)} e^{-n'(h_{\mu^+}-\vp)}.
\end{equation}

\

\ \ \textbf{Step 6. Volume Lemmas for the measures $\Ups$ and $\nu$.}

Now, if $r>0$ is arbitrarily small, there must exist some large integer $n'> n$ such that $$\lambda_2^{(n'+1)(1+\vp)} < r < \lambda_2^{n'(1+\vp)}$$ This implies that, $B(z, \lambda_2^{(n'+1)(1+\vp)}) \subset B(z, r) \subset B(z, \lambda_2^{n'(1+\vp)})$. But when $n'\to \infty$, then we have $r\to 0$, and viceversa. 
Hence if $z \in \Om(n, \vp)$, then from (\ref{gata}) and the last two displayed estimates above, one obtains the following inequalities:
\begin{equation}\label{estpt}
\begin{aligned}
 \frac{\log\lambda_2(\delta_1-\vp) - \log \lambda_1(\delta_1 + \vp) - h_{\mu^+} + \vp}{\log \lambda_2(1+\vp)} &\le \mathop{\lim}\limits_{r\to 0} \frac{\log \Ups(B(z, r))}{\log r} \le \\
& \le \frac{\log\lambda_2(\delta_1+\vp) - \log
 \lambda_1(\delta_1 - \vp) - h_{\mu^+} - \vp}{\log \lambda_2(1+\vp)}.
 \end{aligned}
\end{equation}
But on the other hand, $\Om(n, \vp) \subset \Om(n+1, \vp), n \ge 1$, and thus from (\ref{omups}), 
$$
\Ups\big(\mathop{\cup}\limits_{n \ge 1} \Om(n, \vp) \big) \ge 1- 8M\kappa(\vp).$$
Therefore, if we define the Borel set in $(0, 1) \times (0, 1)$, $$\Om:= \mathop{\cap}\limits_{m \ge 1} \mathop{\cup}\limits_{n \ge 1} \Om(n, \frac 1m),$$
then from the last displayed inequality, and since $\mathop{\lim}\limits_{\vp\to 0} \kappa(\vp) = 0$, it follows that $$\Ups(\Om) = 1.$$ 
In conclusion, from (\ref{estpt}) and (\ref{delta1}), the measure $\Ups$ is exact dimensional and for all $z\in \Om$,
$$  \delta(\Ups)(z) =   \mathop{\lim}\limits_{r\to 0} \frac{\log \Ups(B(z, r))}{\log r} = \frac{h_{\mu^+} - h_{\mu_1}(1-\frac{\chi_2}{\chi_1})}{\chi_2}.
$$

Now we use the exact dimensionality of the projection measure $\Ups = \rho_*\mu^+$ proved above, and the exact dimensionality of the conditional measures on fibers from Theorem \ref{exactT}, together with Theorem 8.7 of \cite{MU-ETDS} applied to $\nu$. Recall that $\mu = \mu_\psi$.  Thus the measure $\nu = \nu_\psi$ is exact dimensional on $X\times Y$. 
Moreover,  from  last displayed formula and (\ref{entplus}), it follows that the Hausdorff (and pointwise) dimension of $\nu$ is given by,
$$HD(\nu) = \delta(\nu) = \delta(\Ups) + \frac{h_\mu}{\chi(\mu)} = \frac{h_{\mu} - h_{\mu_1}(1-\frac{\chi_2(\mu)}{\chi_1(\mu)})}{\chi_2(\mu)} + \frac{h_{\mu}}{\chi(\mu)}.
$$

The case $\lambda_1 \ge \lambda_2$ is proved similarly. 
This concludes  the proof of Theorem \ref{modexT}.

\end{proof}

\

\section{Real-analyticity and Variational Principle for dimension}

Next,   we study equilibrium measures for \textbf{geometric potentials} with respect to endomorphisms $F_T \in \mathcal E$. Consider thus a  conformal Smale skew-product $T$ and let $F_T \in \mathcal E$ given by (\ref{generalF}). 
Define the associated $s$-\textit{geometric potentials} for $s>0$, by
\begin{equation}\label{geomp}
 \psi_{T, s}:\Sigma_I \to \R, \ \ \psi_{T, s}(\eta) := s\log|T_{ \eta|_0^\infty} '(\hat\pi_2\eta)|, \ \text{for} \ \eta \in \Sigma_I,
\end{equation}
where $\hat\pi_2$ was defined in (\ref{pi2hat}).
We show below that $\psi_{T, s}$ is H\"older continuous. By Theorem \ref{eqmesi}, $\psi_{T, s}$ is summable if and only if its pressure $P(\psi_{T, s}) < \infty$.

Define thus  the interval of parameters $s>0$ for which the potential $\psi_{T, s}$ is summable,
\begin{equation}\label{calFT}
\mathcal F(T) := \{s \ge 0, \mathop{\sum}\limits_{i \in I} \sup_{\om\in\Sigma_I^+}\sup_{\xi \in \hat\pi_2[i\om]} |T_{i\om}'(\xi)|^s < \infty\}.
\end{equation} 

 In the sequel, assume that $\psi_{T, s}$ is summable. If $\psi_{T, s}$ is summable on $\Sigma_I$ and H\"older continuous, then by Theorem \ref{eqmesi}, $\psi_{T, s}$ has a unique equilibrium measure $\mu_{T, s}$ on $\Sigma_I$. 

Denote the fiber limit sets by $J_{T, \om}, $   $\om \in  \Sigma_I^+$. 
Recall also the definition of the Lyapunov exponent $\chi(\mu_{T, s})$ from (\ref{lyapglob}).
Let the measure on $\Sigma_I^+$, $$\mu_{T, s}^+ := \pi_{0*}\mu_{T, s},$$  where $\pi_0:\Sigma_I \to \Sigma_I^+$ is the canonical truncation map $\pi_0(\eta) = \eta|_0^\infty, \eta \in \Sigma_I$. For $\mu_{T, s}^+$-a.e $\om \in \Sigma_I^+$, let $\nu_{T, s}^\om$ be the projection measure on $J_{T, \om}$ of the conditional measure $\mu_{T, s}^\om$ of $\mu_{T, s}$, $$\nu_{T, s}^\om = \hat\pi_{2*}\mu_{T, s}^\om.$$

\

The next Theorem shows that  the dimension of $\nu_{T, s}^\om$ does not depend on $\om$, and  that this dimension value depends \textbf{real-analytically} on the parameter $s$.

\

\begin{thm}\label{geoms}
Let a conformal Smale skew-product $T$ and the endomorphism $F_T$ defined in (\ref{generalF}), and let $\psi_{T, s}$ as in (\ref{geomp}).  
Then:

\ a) If $s \in \mathcal F(T)$ then $\nu_{T, s}^\om$ is exact dimensional on $J_{T, \om}$, and  for $\mu_{T, s}^+$-a.e. $\om\in \Sigma_I^+$, $HD(\nu_{T, s}^\om)$ does not depend on $\om \in \Sigma_I^+$ and   $$\delta_{T, s}:= HD(\nu_{T, s}^\om) = \frac{h(\mu_{T, s})}{\chi(\mu_{T, s})}.$$

\  b) The above dimension value $\delta_{T, s}$ depends real-analytically on the parameter $s$ from  the interval $\mathcal F(T)$.

\end{thm}

\begin{proof}

For part a), recall first from (\ref{pi2hat}) that for $T$ fixed as above and for any $\eta \in \Sigma_I$,  $$\hat\pi_2(\eta) = T_{\eta|_{-1}^\infty}\circ T_{\eta|_{-2}^\infty} \circ\ldots,$$ where $T_\om(w)$ is the fiber map of $T$, and $\om= ((m_0, n_0), (m_1, n_1), \ldots) \in \Sigma_I^+$. 
The limit set $J_{T, \om}$ is equal to $\hat\pi_2([\om]), \ \om \in \Sigma_I^+$. By  (\ref{geomp}) and Definition \ref{confSmale} it follows that $\psi_{T, s}$ is H\"older continuous, since the map $\eta \to \hat\pi_2(\eta)$ is H\"older continuous on $\Sigma_I$.

Denote the set of parameters $s$ for which the potential $\psi_{T, s}$ is summable by $\mathcal F(T)$ as above. From Theorem \ref{eqmesi}, $\mathcal F(T)$ is exactly the set of $s>0$ for which $P(\psi_{T, s}) < \infty$, and one can see from definition that $\mathcal F(T)$ is an interval. 
Hence for any $s \in \mathcal F(T)$,   $\psi_{T, s}$ is summable and H\"older continuous on $\Sigma_I$. Then, by Theorem \ref{exactT}, for $\mu_{T, s}^+$-a.e. $\om \in \Sigma_I^+$,  $$HD(\nu_{T, s}^\om) = \delta_{T, s} := \frac{h(\mu_{T, s})}{\chi(\mu_{T, s})}.$$

For  b), $\mu_{T, s}$ is  equilibrium measure of $\psi_{T, s}$, thus $P(\psi_{T, s}) = h(\mu_{T, s}) + \int\psi_{T, s}d\mu_{T, s}$, so $$h(\mu_{T, s}) = P(\psi_{T, s}) - \int\psi_{T, s} d\mu_{T, s}.$$ But the pressure $P(\psi_{T, s})$ depends real-analytically on $s$ (\cite{Ru82}, \cite{Ru}). Then, from the  Ruelle formula for the derivative of the pressure (\cite{Ru}), we obtain
$$\frac{\partial P(\psi_{T, s} + v \psi_{T, s})}{\partial v}\big|_{v = 0} = \int \psi_{T, s} \ d\mu_{T, s}. $$
Thus the integral $\int\psi_{T, s} \ d\mu_{T, s}$ and from above also the entropy $h(\mu_{T, s})$, both depend real-analytically on $s$. So from a), the dimension $\delta_{T, s} = \frac{h(\mu_{T, s})}{\chi(\mu_{T, s})}$  depends real-analytically on the parameter $s$ in the interval $\mathcal F(T)$.  

\end{proof}

Now we prove that the dimension of the fractal set $J_{T, \om}$ is independent of $\om \in \Sigma_I^+$. 

\begin{thm}\label{indepcx}
In the above notation, $HD(J_{T, \om})$ does not depend on $\om\in \Sigma_I^+$.
\end{thm}

\begin{proof}
We have the conformal Smale skew-product $T = (T_\om)_{\om\in \Sigma_I^+}$. From definition it follows that the uniform geometry condition is satisfied and we can apply Theorem 7.2 of \cite{MU-ETDS}. 
Recall also that $J_{T, \om} = \hat\pi_2([\om]) \subset \C$. 
Therefore, for any $\om \in \Sigma_I^+$, a Bowen-type formula holds for the dimension of the set $J_{T, \om}$, namely
\begin{equation}\label{indh}
HD(J_{T, \om}) = \inf\{s\ge 0, P(\psi_{T, s})\le 0\}.
\end{equation}
Therefore, (\ref{indh}) implies that $HD(J_{T, \om})$ does not depend on $\om \in \Sigma_I^+$.

\end{proof}

Since we showed that the dimension of $J_{T, \om}$ does not depend on $\om \in \Sigma_I^+$, denote  by $$\delta_T: = HD(J_{T, \om}), \om \in \Sigma_I^+.$$

\

Next, we will establish a \textbf{Variational Principle for dimension} on  $J_{T, \om}$ in terms of the dimensions of invariant measures supported on  sets $J_{T, \tau}, \tau \in \Sigma_I^+$. 

\

\begin{thm}\label{varprin}

In the above setting, we have  for any $\om \in \Sigma_I^+$, $$\delta_T = HD(J_{T, \om})  = \mathop{\sup}\limits_{s\in \mathcal F(T)} \delta_{T, s} = \sup\{HD(\nu),  \ \nu \text{ invariant measure on} \ J_{T, \tau}, \tau \in \Sigma_I^+\}.$$  
\end{thm}

\begin{proof}

We proved in Theorem \ref{indepcx} that the dimension of $J_{T, \om}$ does not depend on $\om \in \Sigma_I^+$, and thus we denoted it above by $\delta_T$. 
  
 Let  a small $\vp>0$; then from (\ref{indh}) there is $s = s(\vp)>0$ so that $P(\psi_{T, s}) \le 0$ and 
\begin{equation}\label{shd}
s-\vp < \delta_T = HD(J_{T, \om}) \le s.
\end{equation}

But  $P(\psi_{T, s'}) > P(\psi_{T, s})$ for any $s'<s$ from $\mathcal F(T)$, and  the pressure function $s \to P(\psi_{T, s})$ is Lipschitz continuous on $\mathcal F(T)$. Hence there exists a number $\beta(\vp)<0$ with $\beta(\vp) \mathop{\to}\limits_{\vp \to 0} 0$, such that for $s = s(\vp)$ found above,  we have $$P(\psi_{T, s}) > \beta(\vp).$$ Therefore $P(\psi_{T, s}) = h(\mu_{T, s}) - s\chi(\mu_{T,  s}) > \beta(\vp), \ \text{hence}, $
$$\delta_{T, s} = \frac{h(\mu_{T, s})}{\chi(\mu_{T, s})} > s  + \frac{\beta(\vp)}{\chi(\mu_{T, s})}.$$
However from (\ref{lyapglob}) and using the properties of $T$, it follows that  there exists a number $\chi_0>0$, such that $\chi(\mu_{T, s}) > \chi_0$ for all $s \in \mathcal F(T)$. Hence from (\ref{shd}), $$HD(J_{T, \om}) \le s \le \delta_{T, s} + \big|\frac{\beta(\vp)}{\chi_0}\big|$$ 
But $\beta(\vp) \mathop{\to}\limits_{\vp \to 0} 0$, hence $\delta_T = HD(J_{T, \om}) \le \mathop{\sup}\limits_{s \in \mathcal F(T)} \delta_{s}$. \

By the equality $P(\psi_{T, s}) = h(\mu_{T, s}) - s\chi(\mu_{T, s}) \le 0$ and (\ref{shd}) we obtain, for $s = s(\vp)$ above,  $$\delta_{T, s} \le s \le HD(J_{T, \om}) + \vp.$$

From the last inequalities and since  $HD(J_{T, \om})$ is independent of $\om \in \Sigma_I^+$ (from Theorem \ref{indepcx}), we obtain,
\begin{equation}\label{hds}
\delta_T = HD(J_{T, \om}) = \sup \{\delta_{T, s}, \ s \in \mathcal F(T)\}.
\end{equation}

But since any measure on $J_{T, \tau}$ has Hausdoff dimension smaller than or equal to $HD(J_{T, \tau}) = \delta_T$, for any $\tau \in \Sigma_I^+$, we obtain the conclusion of the Theorem. 

\end{proof}

\

\end{document}